\documentclass{article}[12]
\usepackage[utf8]{inputenc}
\usepackage[english]{babel}
\usepackage{mathtools}
\usepackage{capt-of}
\usepackage{xcolor}
\usepackage{amsthm}
\usepackage{tikz}
\usetikzlibrary{positioning}
\usepackage[margin=1in]{geometry}

\newtheorem{Theorem}{Theorem}
\newtheorem{Proposition}[Theorem]{Proposition}
\newtheorem{Lemma}[Theorem]{Lemma}
\newtheorem{Remark}[Theorem]{Remark}
\newtheorem{Definition}[Theorem]{Definition}

\newtheorem{Corollary}[Theorem]{Corollary}
\newtheorem{Example}[Theorem]{Example}

\providecommand{\keywords}[1]
{
  \textbf{\textit{Keywords---}} #1
}

\providecommand{\subjclass}[1]
{
  \textbf{\textit{2020 Mathematics Subject Classification ---}} #1
}

\begin{document}
\title{Neighborly partitions and the numerators \\ of Rogers-Ramanujan identities}
\author{Zahraa Mohsen, Hussein Mourtada}
\date{}

\maketitle

\begin{abstract}
    We prove two partition identities which are dual to the Rogers-Ramanujan identities. These identities are inspired by (and proved using) a correspondence between three kinds of objects: a new type of partitions (neighborly partitions), monomial ideals and some infinite graphs. 
\end{abstract}

\keywords{Rogers-Ramanujan identities, Integer partitions, monomial ideals, simple graphs, Hilbert series.}

\subjclass{11P84, 11P81,05A17,05A19,05C31,13F55,13D40.}
\section{Introduction}

Among the most famous and ubiquitous formulas involving $q-$series, we find the Rogers-Ramanujan identities:

\begin{eqnarray}    
\sum_{k=0}^\infty\frac{q^{k^2}}{(1-q)\cdots(1-q^k)}&=&\frac{\prod\limits_{j\geq 1,~ j~\equiv~ 0,\pm{2}~\mbox{mod}~5}(1-q^j)}{\prod\limits_{j\geq 1}(1-q^j)},\label{rr1}\\
\sum_{k=0}^\infty\frac{q^{k^2+k}}{(1-q)\cdots(1-q^k)}&=&\frac{\prod\limits_{j \geq 2,~ j~\equiv~ 0,\pm{1}~\mbox{mod}~5}(1-q^j)}{\prod\limits_{j \geq 2}(1-q^j)}\label{rr2}.
\end{eqnarray}

\noindent On the left side of the identities, the term corresponding to $k=0$ is taken to be $1.$ Usually the right side of this identity is written after the obvious simplifications which allow to put $1$ in the numerators. We chose this form, because the numerators play an important role in this article. As revealed by MacMahon \cite{McMahon}, these identities can be stated in the realm of the theory of partitions: recall first that an integer partition of a positive integer $n$ is simply a decreasing sequence of positive integers $\lambda=(\lambda_{1},..,\lambda_{r})$, such that $|\lambda|:=\lambda_{1}+\cdots+\lambda_{r}=n$. The $\lambda_{i}$'s are called the parts of $\lambda$ and $r=:size(\lambda)$ is its size; see \cite{A76,Berndt} for more about the theory of partitions. 

\begin{Theorem}\label{RR}(Rogers-Ramanujan identities) Let $n$ be a positive integer. For $i\in \{1,2\}$, let 
$\mathcal{T}_i(n)$ be the set of partitions of $n$ without equal nor consecutive parts and the part $1$ appears at most $i-1$ times. Let $\mathcal{E}_i ( n )$
be the number of partitions of $n$ into parts congruent to $\pm{( 2 + i )}$ mod $5.$ Then we have

$$\mid \mathcal{T}_i(n) \mid=\mid\mathcal{E}_i(n)\mid.$$

\end{Theorem}

\noindent The notation $\mid A \mid$ in the theorem stands for the cardinal of a set $A.$ Observe that the sums side in the $q-series$ identities corresponds to the generating series of the $\mid \mathcal{T}_i(n) \mid$ (first for $i=
2$ and then for $i=1$) and the products side corresponds to  the generating series of the $\mid \mathcal{E}_i(n) \mid.$
These identities appear in many fields other than combinatorics, such as statistical mechanics, number
theory, representation theory, algebraic geometry, probability theory or commutative algebra \cite{Aq,Bax,ASW,GIS,Fu,BMS1,BMS2,GOW,AM,A1,A2,ADJM1,ADJM2}.

The main goal of this article is to prove two identities which are in some sense dual to the Rogers-Ramanujan identities. We begin by introducing the notions appearing in these new identities: neighborly partitions, signature of a neighbourly partition. 

\begin{Definition}
For $i\in \{1,2\}$ and for a positive integer $n,$ we call neighborly partitions of $n$ the set $\mathcal{N}_i(n)$ of partitions $\lambda$ of $n$  which satisfy the following properties:

\begin{enumerate}
\item For every part $\lambda_j$ of $\lambda,$ there exists $l \in \textbf{N}_{>0},~l\not=j$ such that $\mid\lambda_l - \lambda_j\mid \leq 1.$
\item The partition $\lambda$ has at most two equal parts.
\item For every  $ l \in \textbf{N},~~\lambda_l\geq 3-i$ (i.e. for $i=1$, there are no parts  equal to $1).$  
\end{enumerate}

\end{Definition}
The terminology $neighborly$ is inspired by the property 1. which says that for every part $\lambda_j$ of $\lambda,$ there is a \textit{neighbor} part $\lambda_l$ of $\lambda$ which is equal or at a distance $1$ to $\lambda_j.$  As an example, the  integer $4$ has $5$ partitions:  $$4=4=3+1=2+2=2+1+1=1+1+1+1. $$ 

Then the neighborly partitions are $$ \mathcal{N}_1(4)=\{2+2\},~~\mathcal{N}_2(4)=\{2+2,2+1+1\}. $$
Note that the multiplicity (i.e. the number of occurrences) of a part of a neighborly partition is at most two. 
With a neighborly partition $\lambda$, we associate a graph $G_{\lambda}$ as follows: The set $V(G_\lambda)$ of vertices of $G_\lambda$ is in bijection with the set of parts of $\lambda;$ if $h$ is a part of $\lambda$ of multiplicity $1,$ the associated vertex is called $x_{h};$  if $h$ is a part of $\lambda$ of multiplicity $2,$ the vertices associated with the two equal parts are named respectively $x_h$ and $y_h.$ The set $E(G_\lambda)$ of edges of $G_\lambda$ is given by $$E(G_\lambda)=\{(x_{h+1},x_h),(x_l,y_{l}),~~\mbox{for every}~~x_{h+1},x_h,x_l,y_l \in V(G_\lambda) \}. $$

\noindent For example $V(G_{2+1+1})= \{x_2,x_1,y_1\},$ $E(G_{2+1+1})=\{(x_2,x_1),(x_1,y_1)\}$ and $G_{2+1+1}$ has the following shape

\begin{center} 
\begin{tikzpicture}[roundnode/.style={draw,shape=circle,fill=black,minimum size=1mm}]
        \node[roundnode,label={$x_1$}] (x1) {};
        \node[roundnode,label={below:$y_1$}] (y1)[below=of x1] {};
        \node[roundnode,label={above:$x_2$}] (x2)[right=of x1] {};
            
        \draw[] (x1) -- (y1);
        \draw[] (x1) -- (x2);

\end{tikzpicture}\captionof{figure}{$G_{2+1+1}$}

\end{center}
A  subgraph $H$ of $G_{\lambda}$ is said to be vertex-spanning if $V(H)=V(G_{\lambda});$ it is said to be without isolated vertices if any vertex of $H$ is an endpoint of some edge in $E(H).$  

\begin{Definition}
Let $\lambda$ be a neighborly partition. We define the signature $\delta(\lambda)$ of $\lambda$ by

$$\delta(\lambda)=\sum_H(-1)^{\mid E(H)\mid},$$
where $H$ ranges over the vertex spanning subgraphs of $G_\lambda$, which have no isolated vertices and $\mid E(H)\mid$, as mentioned before, is the cardinal of $E(H).$
\end{Definition}


We now are ready to state the main theorem. Let $n$ be a positive integer; for $i \in \{1,2\},$ let $\mathcal{R}_i(n)$ be the set of partitions of $n$ whose parts are larger or equal to $3-i,$ distinct and congruent to $0$ or $\pm{(i)}~$ mod $~5.$

\begin{Theorem} \label{theo} Let $n$ be a positive integer. For $i \in \{1,2\},$ we have the identities

$$\sum_{\lambda \in \mathcal{N}_i(n) }\delta(\lambda)=\sum_{\lambda \in \mathcal{R}_i(n)}(-1)^{size(\lambda)}. $$

\end{Theorem}

\begin{Example}
\begin{enumerate}
    \item For $n=6$ and $i=1$, we have $\mathcal{N}_1(6)=\{3+3\}$ and $\mathcal{R}_1(6)=\{6\}.$ The graph 
    $G_{3+3}$ has the following shape
    
  \begin{center}\begin{tikzpicture}[roundnode/.style={draw,shape=circle,fill=black,minimum size=4mm}]
        \node[roundnode,label={$x_3$}] (x3) {};
        \node[roundnode,label={below:$y_3$}] (y3)[below=of x3] {} ;
        
        \draw[] (x3) -- (y3);
  \end{tikzpicture}\captionof{figure}{$G_{3+3}$}
  \end{center}
  and  it has no (strict) subgraphs which are vertex-spanning without isolated singularities. Hence, we have
  
  $$\sum_{\lambda \in \mathcal{N}_1(6) }\delta(\lambda)=\delta(3+3)=(-1)^{1}=\sum_{\lambda \in \mathcal{R}_1(6)}(-1)^{size(\lambda)}=(-1)^{size(6)}=(-1)^{1}.$$
\item For $n=6$ and $i=2$, we have $\mathcal{N}_2(6)=\{3+3,3+2+1,2+2+1+1\}$ and $\mathcal{R}_2(6)=\phi.$ The graphs corresponding to the partitions $3+3,3+2+1$ have no (strict) strict vertex-spanning subgraphs without isolated singularities ; hence we have $\delta(3+3)=(-1)^1$ and $\delta(3+2+1)=(-1)^2$. The graph $G_{2+2+1+1}$ has the shape

  \begin{center}
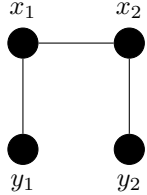
\begin{tikzpicture}[roundnode/.style={draw,shape=circle,fill=black,minimum size=4mm}]
        \node[roundnode,label={$x_1$}] (x1) {};
        \node[roundnode,label={below:$y_1$}] (y1)[below=of x1] {};
        \node[roundnode,label={above:$x_2$}] (x2)[right=of x1] {};
        \node[roundnode,label={below:$y_2$}] (y2)[below=of x2] {};  
        \draw[] (x1) -- (y1);
        \draw[] (x1) -- (x2);
        \draw[] (x2) -- (y2);
\end{tikzpicture}\\\captionof{figure}{$G_{2+2+1+1}$}

\end{center}

and beside $G_{2+2+1+1}$ itself, its only vertex-spanning subgraph without isolated singularities is

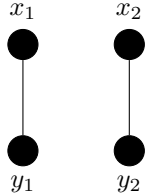
\begin{center}\begin{tikzpicture}[roundnode/.style={draw,shape=circle,fill=black,minimum size=4mm}]
        \node[roundnode,label={$x_1$}] (x1) {};
        \node[roundnode,label={below:$y_1$}] (y1)[below=of x1] {};
        \node[roundnode,label={above:$x_2$}] (x2)[right=of x1] {};
        \node[roundnode,label={below:$y_2$}] (y2)[below=of x2] {};  
        \draw[] (x1) -- (y1);
        \draw[] (x2) -- (y2);
\end{tikzpicture}\\

\end{center}
Hence, we have $\delta(2+2+1+1)=(-1)^{3}+(-1)^{2}=0$ and  
$$\sum_{\lambda \in \mathcal{N}_2(6) }\delta(\lambda)=\delta(3+3)+\delta(3+2+1)+\delta(2+2+1+1)=-1+1+0=0=
  \sum_{\lambda \in \mathcal{R}_2(6)}(-1)^{size(\lambda)}.$$

\end{enumerate}
\end{Example}

For $i \in \{1,2\},$ We set $\mathcal{N}_i=\cup_{n \in \textbf{N}_{>0}}  \mathcal{N}_i(n).$ By considering the generating sequence of both sides of Theorem \ref{theo}, we find the following equivalent statement:

\begin{Corollary} \label{cor}  For $i \in \{1,2\},$ we have the identities
$$\sum_{\lambda \in \mathcal{N}_i }\delta(\lambda)q^{\mid \lambda \mid}=
\prod_{j\geq 3-i,j~\equiv~ 0,\pm{i}~\mbox{mod}~5}(1-q^j).$$
\end{Corollary}

One observes that the right side of the identities in corollary \ref{cor} corresponds to the numerators of the identities (\ref{rr1}) and (\ref{rr2}). As we will show later, the left side of the identities in the theorem is related to the left side 
of the identities (\ref{rr1}) and (\ref{rr2}). Here we use the Rogers-Ramanujan identities to prove theorem \ref{theo}. The proof of the main theorem is actually a proof that the Rogers-Ramanujan identities are equivalent to corollary
\ref{cor}; in particular, a direct proof (i.e. which does not use the Rogers-Ramanujan identities) of theorem \ref{theo} would also give a proof of the Rogers-Ramanujan identities.

\section{Strategy and proof of the main result}

We will divide the scheme of the proof into three steps which are the subject of the three subsections of this section.

\begin{enumerate}
    \item For $i\in \{1,2\},$ we interpret the $q-$series $\sum_{\lambda \in \mathcal{N}_i }\delta(\lambda)q^{\mid \lambda \mid}$ (see corollary \ref{cor}) in terms of a generating series $S_i$ "counting" the subgraphs of an infinite (simple) graph $G_i^{\infty}.$
    \item For $i\in \{1,2\},$ we give a formula relating $S_i$ to the Hilbert series $\mathcal{H}_i$ of a graded ring which is the quotient of polynomial ring with a (infinite) countable number of variables: In this step, we use the notion of edge ideals, and the result that we prove is valid for any edge ideal.
    \item For $i\in \{1,2\},$ we describe a formula relating $\mathcal{H}_i$ in terms of the generating series of  $\mid \mathcal{T}_i(n) \mid$ (see theorem \ref{RR}). Then we use the Rogers-Ramanujan identities to obtain the main results.
    
\end{enumerate}

All the graphs that we will consider are simple, i.e they do not have more than one edge between any two vertices and no edge starts and ends at the same vertex.

\subsection{Neighborly partitions and enumerating subgraphs}
For $i\in \{1,2\},$ we begin by considering the infinite graph $G_i^{\infty}$ whose set  of vertices is 
$V(G_i^{\infty})=\{x_j,y_j;j\in \textbf{N},j\geq 3-i\}$ and whose set of edges is $E(G_i^{\infty})=\{(x_j,x_{j+1}),(x_j,y_j);j\geq 3-i\}.$ Notice that for a neighborly partition $\lambda \in \mathcal{N}_i,$ the graph $G_\lambda$ is an \textit{induced} subgraph of $G_i^{\infty}$ which has no isolated vertices. 
 So $G_i^{\infty}$ has the following shape
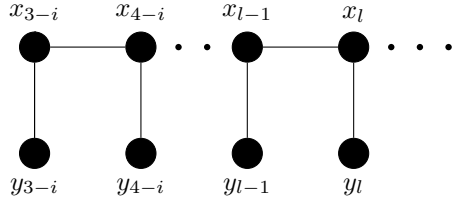
\begin{center}\begin{tikzpicture}[roundnode/.style={draw,shape=circle,fill=black,minimum size=4mm}]
        \node[roundnode,label={$x_{3-i}$}] (x2) {};
        \node[roundnode,label={below:$y_{3-i}$}] (y2)[below=of x2] {};
        \node[roundnode,label={above:$x_{4-i}$}] (x3)[right=of x2] {};
        \node[roundnode,label={below:$y_{4-i}$}] (y3)[below=of x3] {};
        \node[roundnode,label={above:$x_{l-1}$}] (xl-1)[right=of x3] {};
        \node[roundnode,label={below:$y_{l-1}$}] (yl-1)[below=of xl-1] {};
         \node[roundnode,label={above:$x_l$}] (xl)[right=of xl-1] {};
        \node[roundnode,label={below:$y_l$}] (yl)[below=of xl] {};
        \draw[] (x1) -- (y1);
        \draw[] (x1) -- (x2);
        \draw[] (x2) -- (y2);
        \draw[] (x2) -- (x3);
        \draw[] (x3) -- (y3);
        \draw[] (xl) -- (yl);
         \draw[] (xl-1) -- (xl);
        \draw[] (xl-1) -- (yl-1);
        \path (x3) -- (xl-1) node [right=0.5mm of x3,black, font=\Huge] {$\dots$};  
        \node[right=0.5mm of xl,black, font=\Huge] {$\dots$};
\end{tikzpicture}\end{center}

Recall that a subgraph $H$ is said to be induced if an edge of $G$ is an edge of $H$ whenever its endpoints are vertices of $H.$ In particular, such a subgraph is completely determined by it vertices. Conversely any induced subgraph of $G_i^{\infty},i=1,2, $  without isolated singularities is of the type $G_\lambda$
for some $\lambda \in \mathcal{N}_i.$ 

\begin{Definition}
Let $G$ be a simple graph, let $V(G)=\{v_j,j\in I\}$ be its set of vertices that we assume countable. We call the multivariable subgraph enumerating series  of $G$ the series in the variables $(v_j)_{j \in I}$ and which is defined by

$$S_G(\textbf{v},z)=\sum_{H}(\prod_{v_j\in V(H)}v_j)z^{\mid E(H) \mid},$$
where $H$ ranges over finite subgraphs of $G$ without isolated vertices and where we denoted by $\textbf{v}$ the multivariable $(v_j)_{j\in I}.$

\end{Definition}

Note that, since we are only considering finite subgraphs, the monomials in the multivariable subgraph enumerating series make intervene a finite number of variables. We will express the series corollary \ref{cor} in terms of the
multivariable subgraph enumerating series of $G_i^{\infty},i=1,2.$ Recall the vertices $v_j$ of $G_i^{\infty},i=1,2$ belong to $\{x_j,y_j;j\in \textbf{N},j\geq 3-i\}.$ Denote by 
$\textbf{x}$ the multivariable $(x_j)_{j\geq 3-i}$ and by $\textbf{y}$ the multivariable $(y_j)_{j\geq 3-i}.$  
\begin{Lemma}\label{enum} For $i \in \{1,2\},$ set $S_i(\textbf{v},z):=S_{G_i^{\infty}}(\textbf{v},z)=S_i(\textbf{x},\textbf{y},z)$. Let $S_i^w(q,z)$ be the series which is obtained from $S_i(\textbf{x},\textbf{y},z)$ by substituting $x_j$
and $y_j$ by $q^j.$ We have 

$$\sum_{\lambda \in \mathcal{N}_i }\delta(\lambda)q^{\mid \lambda \mid}=S_i^w(q,-1).$$

\end{Lemma}

\begin{proof}
For $i\in \{1,2\},$ we have $$S_i(\textbf{v},z)=\sum_{H}(\prod_{v_j\in V(H)}v_j)z^{\mid E(H) \mid}$$
$$=\sum_{V\subset V(G_i^{\infty})}(\sum_{\{H;V(H)=V\}}(\prod_{v_j\in V}v_j)z^{\mid E(H) \mid})=\sum_{V\subset V(G_i^{\infty})}(\prod_{v_j\in V}v_j)(\sum_{\{H;V(H)=V\}}z^{\mid E(H)\mid}), $$
 where $H$ ranges over finite subgraphs of $G_i^{\infty}$ without isolated vertices and $V$ over the sets of vertices of such subgraphs $H.$ So we have gathered the subgraphs which have the same set of vertices $V$ together. The induced subgraph of $G_i^{\infty}$ whose set of vertices is such a $V$ is of the form $G_{\lambda}$ for some $\lambda \in \mathcal{N}_i.$ In particular, if we replace in $\prod_{v_j\in V}v_j$ the $x_j$'s and the $y_j$'s by $q^j,$ we obtain $q^{|\lambda|}.$  
 Moreover if we set $\varphi(z)=\sum_{\{H;V(H)=V\}}z^{\mid E(H)\mid},$ we obtain by the definition of the signature that $\delta(\lambda)=\varphi(-1)$ and the lemma follows.

\end{proof}

\subsection{Multigraded Hilbert series of Edge ideals and enumerating subgraphs}
Let $G$ be a simple graph, let $V(G)=\{v_j,j\in I\}$ be its set of vertices that we assume countable or finite. Let $E(G)$ be set of edges of $G$ that we also assume countable or finite (this is sufficient for our purposes).
Let $\textbf{K}$ be a field of zero characteristic. We consider the ring of polynomials $A=\textbf{K}[\textbf{v}]=\textbf{K}[v_j, j\in I]$ whose variables range in the set of vertices $V(G).$ With the graph $G$ we associate its edge ideal which is a square-free monomial ideal generated by
$$\mathcal{I}(G)=<v_jv_l \mid (v_j,v_l) \in E(G)>.$$
We are interested in the multigraded Hilbert series which counts the monomials in the quotient ring of the polynomial ring by an edge ideal.

\begin{Definition}\label{Mon}
Let $G$ be a simple graph as above. The multigraded Hilbert series of 
$A/\mathcal{I}(G)$ is the series $\mathcal{H}_G$ in the variables $v_j, j\in I$ which is defined by
$$\mathcal{H}_{A/\mathcal{I}(G)}(\textbf{v})=\mathcal{H}_G(\textbf{v})=\sum_{M \not\in \mathcal{I}(G)} M, $$
where $M$ ranges over the monomials of $A$ (making intervene a finite number of variables) which are not in the ideal 
$\mathcal{I}(G).$

\end{Definition}

Note that the multigraded Hilbert series of $A/\mathcal{I}(G)$ "counts" the monomials in $A$ which are not zero in $A/\mathcal{I}(G).$ We have the following formula expressing the multigraded Hilbert series in terms of the multivariable subgraph enumerating series.

\begin{Lemma}\label{hilb}

Let $G$ be a simple graph as above. We have

$$\mathcal{H}_G(\textbf{v})=\frac{S_G(\textbf{v},-1)}{\prod_{j\in I}(1-v_j)}. $$
\end{Lemma}
\begin{proof}
First one remarks that (one may think of the one variable and then two variables cases to be convinced)
$$\mathcal{H}_A(\textbf{v})=\frac{1}{\prod_{j\in I}(1-v_j)}.$$
Now recall that a monomial belongs to a monomial ideal if and only if it is divisible by at least one of its generators. Let $M=\{m_1,m_2,\ldots\}$ be the set of generators of $\mathcal{I}(G)$ which are in bijection with the edges of $G.$ We have the formula
$$ \mathcal{H}_{A/\mathcal{I}(G)}(\textbf{v})=\mathcal{H}_{A}(\textbf{v})
-\sum_{m\in M}m \mathcal{H}_{A}(\textbf{v})$$
$$+\sum_{\{m_{j_1},m_{j_2}\}\subset M}  lcm(m_{j_1},m_{j_2})\mathcal{H}_{A}(\textbf{v})+\cdots+$$ 
$$(-1)^k \sum_{\{m_{j_1},m_{j_2},\ldots,m_{j_k}\}\subset M}  lcm(m_{j_1},m_{j_2},\ldots,m_{j_k})\mathcal{H}_{A}(\textbf{v})+\cdots=
$$ 
$$\mathcal{H}_{A}(\textbf{v})[1-\sum_{m\in M}m+\sum_{\{m_{j_1},m_{j_2}\}\subset M}  lcm(m_{j_1},m_{j_2})+\cdots+$$ 
\begin{equation}\label{incexl}
(-1)^k \sum_{\{m_{j_1},m_{j_2},\ldots,m_{j_k}\}\subset M}  lcm(m_{j_1},m_{j_2},\ldots,m_{j_k})+\cdots],
\end{equation}
where $lcm$ stands for the least common multiple. The formula expresses that the  (non-zero) monomials
$A/\mathcal{I}(G)$ are the monomials in $A$ except those which are divisible by one of the $m_j'$s. When taking out the monomials which are divisible by one of the $m_j$ (which is expressed by $-\sum_{m\in M}m \mathcal{H}_{A}(\textbf{v})$), we take out twice those which are divisible at the same time by some $m_{j_1}$ and $m_{j_2}$ (i.e. by $lcm(m_{j_1},m_{j_2}$)); so we need to add once those which are divisible by both (this is expressed by adding 
$\sum_{\{m_{j_1},m_{j_2}\}\subset M}  lcm(m_{j_1},m_{j_2})\mathcal{H}_{A}(\textbf{v}))$ hence adding twice those which are divisible by three and so on: this is simply the inclusion exclusion-principle since the monomials in $\mathcal{I}(G)$ are those which belong to the union of the ideals generated by the $m_j'$s.  Now remark that choice of $k$ monomials  $m_{j_1},m_{j_2},\ldots,m_{j_k}$ corresponds to the choice of $k$ edges of $G,$ i.e. a subgraph of $G$ having $k$ edges and  the variables appearing in the associated monomial (which are square free, the graph $G$ being simple) correspond to vertices of this subgraph; hence we have $$1-\sum_{m\in M}m+\cdots +(-1)^k \sum_{\{m_{j_1},m_{j_2},\ldots,m_{j_k}\}\subset M}  lcm(m_{j_1},m_{j_2},\ldots,m_{j_k})+\cdots=S_G(\textbf{v},-1)$$
and the proposition follows.

\end{proof} 

\begin{Remark}
\begin{enumerate}
  \item Note that the formula (\ref{incexl}) in lemma \ref{hilb} can also be seen as a direct application of the Taylor resolution (see \cite{Pe}) of the monomial ideal $\mathcal{I}(G).$
\item A variant of the formula in lemma \ref{hilb} exists in the literature, often for special gradings \cite{Go}; however, this exact statement will be needed in the next subsection and the proof makes the paper self-contained.
\end{enumerate}
\end{Remark}

\subsection{Proof of the main result}
For $i=1,2$ we now consider the $\textbf{K}-$algebra $\mathcal{P}_i:=K[x_j,y_j,j\geq 3-i]/\mathcal{I}(G_i^{\infty})$ which, by definition of the ideal $\mathcal{I}(G_i^{\infty})$ is  equal to $$\frac{\textbf{K}[x_j,y_j,j\geq 3-i]}{<x_jy_j,x_jx_{j+1},j\geq 3-i>}.$$
By giving to $x_j$ and $y_j,$ for $j\geq 3-i,$ the weight $j,$ the algebra $\mathcal{P}$ inherits the structure of a graded algebra, \textit{i.e.} we can write $$\mathcal{P}_i=\oplus_{h \in \textbf{Z}_{\geq 0}} \mathcal{P}_{i,h}$$
where the $\mathcal{P}_{i,h}$'s are additive groups satisfying $\mathcal{P}_{i,h}.\mathcal{P}_{i,h'}
\subset \mathcal{P}_{i,h+h'}.$ By abuse of notation, the 
$\mathcal{P}_{i,h}$ is a $\textbf{K}-$ vector space generated by all the monomials which does not belong to $\mathcal{I}(G_i^{\infty})$ and which have weight equal to $h.$ The Hilbert-Poincar\'e series of the graded algebra $\mathcal{P}_i$ is by definition
$$\mathcal{HP}_{\mathcal{P}_i}(q):=\sum_{h \in \textbf{Z}_{\geq 0}}
dim_\textbf{K}\mathcal{P}_{i,h}q^h.$$
One notices that the monomials in the varaibles $x_j,y_j$ which are of weight 
$h$ are exactly those such that when we replace $x_j$ and $y_j$ by $q^j,$ we obtain $q^h:$ for instance, $x_jx_{j'}$ is of weight $j+j'$ and the mentioned substitution gives $q^{j+j'}.$ Hence, if we set $\mathcal{H}^w_{\mathcal{P}_i}(q)$ to be the series which is obtained from 
$\mathcal{H}_{G_i^\infty}(\textbf{v})$ (see Definition \ref{Mon}) by substituting $x_j$ and $y_j$ by $q^j$ we find

$$\mathcal{HP}_{\mathcal{P}_i}(q)= \mathcal{H}^w_{\mathcal{P}_i}(q).$$
Applying lemma \ref{hilb} and lemma \ref{enum} we deduce the following:

\begin{Proposition}\label{HP}
We have $$\mathcal{HP}_{\mathcal{P}_i}(q)=\frac{ \sum_{\lambda \in \mathcal{N}_i }\delta(\lambda)q^{\mid \lambda \mid}}{\prod_{j\geq 3-i}(1-q^j)^2}.$$
\end{Proposition}

To make the link with Rogers-Ramanujan identities we will consider polarization of monomial ideals: this is a procedure that allows to associate with a monomial ideal a squarefree monomial ideal in a polynomial ring which has more variables. The interesting fact is that invariants of both ideals are very related. The procedure amounts to replace any power of the type $x^e$ (in some polynomial ring having $x$ as a variable) by the monomial $xy_1\ldots y_{e-1}$ where the $y_i$ are variables in a new polynomial ring extending the variables in the original polynomial ring. By applying this process to the generators of a monomial ideal, we obtain a squarefree monomial ideal. The important examples for us are the ideals $$\mathcal{I}(G_i^\infty)=<x_jy_j,x_jx_{j+1},j\geq 3-i> \subset \textbf{K}[x_j,y_j,j\geq 3-i],\  \text{for} \ i=1,2.$$
These are the polarization (this is why we called our ring above $\mathcal{P}_i$)
of the ideals 

$$<x_j^2,x_jx_{j+1},j\geq 3-i> \subset \textbf{K}[x_j,j\geq 3-i], $$
where the grading is induced from giving to $x_j$ the weight $j.$ We have the following (Corollary 1.6.3 in \cite{HH}, see also \cite{MiS,Pe}):

\begin{equation}\label{inf}
    \mathcal{HP}_{\mathcal{P}_i}(q)=
\frac{\mathcal{HP}_{\textbf{K}[x_j,j\geq 3-i]/<x_j^2,x_jx_{j+1},j\geq 3-i>}(q)}{\prod_{j\geq 3-i}(1-q^j)}.
\end{equation}

Note that, in \textit{loc. cit.}, the proof is given for finitely generated ideals. But this extends easily to our situation. Indeed,  the $\textbf{K}-$ algebra 
$$R:=\textbf{K}[x_j,j\geq 3-i]/<x_j^2,x_jx_{j+1},j\geq 3-i>$$ is the inductive limit
of the $\textbf{K}-$finitely generated algebra
$$R_n=\textbf{K}[x_j, 3-i\leq j\leq n+1]/<x_j^2,x_jx_{j+1},3-i \leq j \leq n>.$$
Moreover, since the weights of the variables $x_j$ are growing we can see that 
$$\mathcal{HP}_{R_n}(q)=\mathcal{HP}_{R}(q)~~\mbox{modulo}~~q^n.$$
This implies 

$$ \lim\limits_{n \rightarrow +\infty} \mathcal{HP}_{R_n}(q)=\mathcal{HP}_{R}(q),$$
and that the case of finitely generated ideals gives the equality (\ref{inf}).\\

Now everything is settled down for the proof of Corollary \ref{cor} (which is equivalent to Theorem \ref{theo}).

\begin{proof} On one hand, from Proposition \ref{HP} and the equality (\ref{inf}) we otain that for $i\in \{1,2\}$ we have
\begin{equation}\label{avd}
\frac{ \sum_{\lambda \in \mathcal{N}_i }\delta(\lambda)q^{\mid \lambda \mid}}{\prod_{j\geq 3-i}(1-q^j)^2}=\frac{\mathcal{HP}_{\textbf{K}[x_j,j\geq 3-i]/<x_j^2,x_jx_{j+1},j\geq 3-i>}(q)}{\prod_{j\geq 3-i}(1-q^j)}.
\end{equation}

On the other hand, the homogeneous components of weight $h$ of
$$R=\frac{\textbf{K}[x_j,j\geq 3-i]}{<x_j^2,x_jx_{j+1},j\geq 3-i>}$$

is generated by the monomials $x_{i_1}\cdots x_{i_r}$ such that $i_1+ \cdots +i_r=h$ and
 which are not divisible by neither $x_j^2$ nor $x_jx_{j+1},$ for any $j \in \textbf{N}.$ 
 So the data of such a monomial is equivalent to the data of a partitions of $h$ without equal nor consecutive parts and the part $1$ appears at most $i-1$ times (because of the
 condition $j \geq 3-i$ in the indices of the variables $x_j$). Hence we have 

$$\mathcal{HP}_R(q)=\sum_{h \geq 0}\mid \mathcal{T}_i(h)\mid q^h ,$$ which is the left side of the identities (\ref{rr1}) (for $i=2)$ and of (\ref{rr2}) (for $i=1).$ From the equalities (\ref{avd}), (\ref{rr1}) and (\ref{rr2}), we get that for $i \in \{1,2\},$ we have:
$$\frac{ \sum_{\lambda \in \mathcal{N}_i }\delta(\lambda)q^{\mid \lambda \mid}}{\prod\limits_{j\geq 3-i}(1-q^j)}=
\frac{\prod\limits_{j\geq 3-i,~j~\equiv~ 0,\pm{i}~\mbox{mod}~5}(1-q^j)}{\prod\limits_{j \geq 3-i}(1-q^j)},$$

and the theorem follows.

\end{proof}
\textbf{Acknowledgment} The first author would like to acknowledge the National Council for Scientific Research of Lebanon (CNRS-L) and the Agence Universitaire de la Francophonie in cooperation with Lebanese University for granting a doctoral fellowship to her.

\bigskip
\author[Z. Mohsen]{
{Universit\'e de Paris, Sorbonne Universit\'e, CNRS, Institut de Math\'ematiques de Jussieu-Paris Rive Gauche, F-75013 Paris, France 
and Department of Mathematics, Lebanese University, Hadath, Beirut, Lebanon}\\
{zahraa.mohsen@imj-prg.fr}\\

\author[H. Mourtada]
{Universit\'e de Paris, Sorbonne Universit\'e, CNRS, Institut de Math\'ematiques de Jussieu-Paris Rive Gauche, F-75013 Paris, France }\\
{hussein.mourtada@imj-prg.fr}

\end{document}